\documentclass[11pt]{article} 
\usepackage[english]{babel}
\usepackage[T1]{fontenc} 
\usepackage[utf8]{inputenc} 
\usepackage{amsmath,amssymb}
\usepackage{amsthm}
\usepackage{bbm}
\usepackage{mathrsfs}
\usepackage{graphicx} 
\usepackage{xcolor}
\usepackage{MnSymbol,wasysym}
\usepackage{leftidx}
\usepackage{setspace}

\usepackage{graphicx} 
\usepackage{xcolor}
\usepackage{MnSymbol,wasysym}
\usepackage{leftidx}
\usepackage{setspace}

\usepackage[margin=3cm]{geometry}

\newtheorem{theorem}{Theorem}

\newtheorem{remark}[theorem]{Remark}
\newtheorem{lemma}[theorem]{Lemma}

\newtheorem{example}[theorem]{Example}

\usepackage{titlesec}
\usepackage{mathrsfs}
\providecommand{\keywords}[1]{\textbf{\textit{Keywords---}} #1}
\begin{document}
	\title{ On the inverse gamma subordinator}
	\author{Fausto Colantoni \and Mirko D'Ovidio}

	

\newcommand{\Addresses}{{
		\bigskip
		\footnotesize
		
		Department of Basic and Applied Sciences for Engineering, Sapienza University of Rome, Italy
		\par\nopagebreak
		\textit{E-mail address} \texttt{\{fausto.colantoni,mirko.dovidio\}@uniroma1.it}
		
		\medskip

}}

\maketitle
\begin{abstract}
	In this paper we deal with some open problems concerned with gamma subordinators. In particular, we provide a representation for the moments of the inverse gamma subordinator. Then, we focus on $\lambda$-potentials and we study the governing equations associated with gamma subordinators and inverse processes.
	Such representations are given in terms of \textit{higher transcendental functions}, also known as \textit{Volterra functions}
\end{abstract}
\section{Introduction}

Gamma subordinator is a well-known subordinator which has been considered in many fields of Applied Sciences. In Mathematical Finance for instance, a well-known process is the Variance Gamma Process (or \textit{Laplace motion}) which can be obtained by considering a Brownian motion with a random time given by a Gamma subordinator (\cite{levy},\cite{beghin2015fractional},\cite{lapmotion},\cite{VGoption}). We recall that a subordinator is a Lévy process with non-negative and non-decreasing paths. The Gamma subordinator is a special case in which the associated Lévy measure on $(0,\infty)$ is infinite, then the paths are increasing (\cite{sato}). 
 
In the connection between non-local analysis and Probability the Gamma subordinator plays a relevant role, many authors have investigated such a connection and the related properties, we list only few references throughout the work. Recently, in \cite{beghin} some new operators associated with Gamma subordinators appear whereas, in \cite{delrus} the connection between parabolic and elliptic problems in case of Gamma (and inverse Gamma) time change is considered. 

Despite the fact that Gamma subordinators are well-known objects, deeply investigated in the past years, there is a lack in the theory concerned with inverse Gamma subordinators. At the current stage there are some results on
the moments of the inverse gamma subordinators only concerned with their
 asymptotic behavior (\cite{kumarPS},\cite{delrus}) and their Laplace transforms (\cite{kumarPS}, \cite{kumar2},\cite{veillette}).

Our main contributions, in order to close such a gap, are stated in Section \ref{sec:MainResults}. We obtain a new representation for the moments of any real (positive) order. Moreover, we obtain explicitly, in a closed form, the potentials and the Sonine kernels concerned with our case study, that is for Gamma subordinators and their inverses. We also provide a detailed discussion on densities and governing equations in order to have a clear picture about the processes we deal with. Indeed, the Gamma subordinator belongs to a special class of time-dependent continuous functions. 

In our analysis a central role has been played by the functions $\nu$ and $\mu$, two \textit{higher transcendental functions} (see for example the book \cite{EHIII}), also known as \textit{Volterra functions} (\cite{apelblat},\cite{mainardivolterra}). The function $\nu$  has been introduced by Volterra (1916) in his theory of convolution-logarithms (\cite{volterra}), hence the name. The importance of these functions does not surprise, indeed they seem to be the analogue of the Mittag-Leffler function in case of stable subordinator and the corresponding inverses. The interest readers can consult the book \cite{mittag} for the Mittag-Leffler function.

\section{Preliminaries}
\subsection{Gamma subordinators}
\label{Sec:GammaSub}


We introduce the Bernstein function $\Phi: (0, \infty) \mapsto (0, \infty)$ which is uniquely defined by the so-called Bernstein representation
\begin{align*}
\Phi(\lambda)=\int_0^\infty (1-e^{-\lambda z}) \Pi(dz) , \quad \lambda \geq 0
\end{align*}
where $\Pi$ on $(0,\infty)$ with $\int_0^\infty (1 \wedge z) \Pi(dz) <\infty$ is the associated Lévy measure. We also recall that
\begin{align}
\label{lapTailLm}
\frac{\Phi(\lambda)}{\lambda}=\int_0^\infty e^{-\lambda z} \overline{\Pi}(z) dz,
\end{align}
where $\overline{\Pi}(z)=\Pi((z,\infty))$ is termed \textit{tail of the Lévy measure} (see \cite{bertoin} Section 1.2 for details).\\

From now on we consider the symbol
\begin{align}
\label{symbGamma}
\Phi(\lambda) = a \ln \left(1 + \frac{\lambda}{b} \right) = a \int_0^\infty \left(1 - e^{- \lambda y} \right) \frac{e^{- b y}}{y}\, dy, \quad \lambda \geq 0, \quad a>0, \; b>0
\end{align}
and the associated gamma subordinator $H=\{H_t\}_{t\geq 0}$ for which 
\begin{align}
\label{LapH}
\mathbf{E}_0[e^{-\lambda H_t}] = e^{-t \Phi(\lambda)}, \quad \lambda \geq 0.
\end{align}
Since $\Pi(0,\infty)=\infty$, then from Theorem 21.3 of \cite{sato}, we have that $H_t$ has increasing sample path with jumps. The inverse process
\begin{align*}
L_t := \inf\{s>0\,:\, H_s \notin (0, t) \}, \quad t>0
\end{align*}
can be regarded as an exit time for $H$. In particular, the process $L=\{L_t\}_{t\geq 0}$ turns out to be non-decreasing with continuous paths and it can be associated with some delaying or rushing effects (see \cite{delrus}). By definition of inverse process, we can also write
\begin{align}
\label{inverseRelation}
\mathbf{P}_0(H_t<s)=\mathbf{P}_0(L_s>t) , \ \ \ s,t>0
\end{align}
where $\mathbf{P}_x$ denote the probability measure for a process started from $x$ at time $t=0$ and $\mathbf{E}_x$ the mean value with respect to $\mathbf{P}_x$. Here we have that $L_0=0$ and $H_0=0$. We also use the notation
\begin{align*}
\mathbf{P}_0(H_t \in dx) = h(t,x)\,dx, \quad \mathbf{P}_0(L_t \in dx) = l(t,x)\, dx.
\end{align*}
It is well known that, $\forall\, t > 0$,
\begin{align}
\label{repHsing} 
h(t,x) 
= & \left\lbrace
\begin{array}{ll}
\displaystyle \frac{b^{at}}{\Gamma(at)} x^{at-1} e^{-bx}, & x>0\\
\displaystyle  0, & x \leq 0
\end{array}
\right.
\end{align}
verifies
\begin{align*}
\mathbf{P}_0(H_t < 0) =0 \quad \textrm{and} \quad \mathbf{P}_0(H_t\geq 0 ) =1 \quad \forall t \geq 0
\end{align*}
and
\begin{align}
\label{laph}
\displaystyle \int_0^\infty e^{-\lambda x} h(t,x)\, dx = \frac{b^{at}}{(\lambda + b)^{at}} = e^{- t\, a\,\ln\left(1+ \frac{\lambda}{b}\right)}, \quad \lambda>0, \ t >0
\end{align}
which agrees with formula \eqref{LapH}. However, we immediately see that
\begin{align*}
h(t,x) \to 0 \quad \textrm{as} \quad x\downarrow 0 \quad \textrm{only if} \quad at>1.
\end{align*}
The property for a Lévy process to be continuous and derivable  depending on the time variable is known as \textit{time dependent property} (see \cite[Chapter 23]{sato}). Such a property turns out to be quite demanding from the analytical point of view. We conclude the discussion about $H$ by recalling that, from the representation \eqref{repHsing}, we are able to evaluate the moments
\begin{align}
\label{momentsH}
	\mathbf{E}_0[(H_t)^q]=\frac{1}{b^q} \frac{\Gamma(at+q)}{\Gamma(at)}, \quad q \geq 1, \; t>0.
\end{align}

In \cite[formula (17)]{kumarPS}, the authors provide the following representation for the density of $L$, 
\begin{align*}
l(t,x)=\frac{a e^{-t}}{\pi} \int_0^\infty \frac{y^{-ax} e^{-yt}}{1+y}\left(\pi\cos(a \pi x)-\ln(y)\sin(a \pi x) \right)\ dy, \ ax \notin \mathbb{N}
\end{align*}
which is written in their work for $b=1$ (our notation).

 Concerning the analogue of formulas \eqref{laph} and \eqref{momentsH}, the explicit representation for the process $L$ is an open problem.

\subsection{Special functions}
Let us introduce the incomplete gamma function
\begin{align}
\label{defGammaInc}
\gamma(a,z)=\int_0^z e^{-y} y^{a-1} dy, \quad  a>0, \; z \geq 0
\end{align}
and the digamma function
\begin{align}
\label{defdigamma}
\psi_0(z)=\frac{1}{\Gamma(z)} \int_0^\infty y^{z-1} \ln y \ e^{-y} dy, \quad z \geq 0.
\end{align}
\begin{figure}[h]
	\centering
	\includegraphics[width=5.5cm]{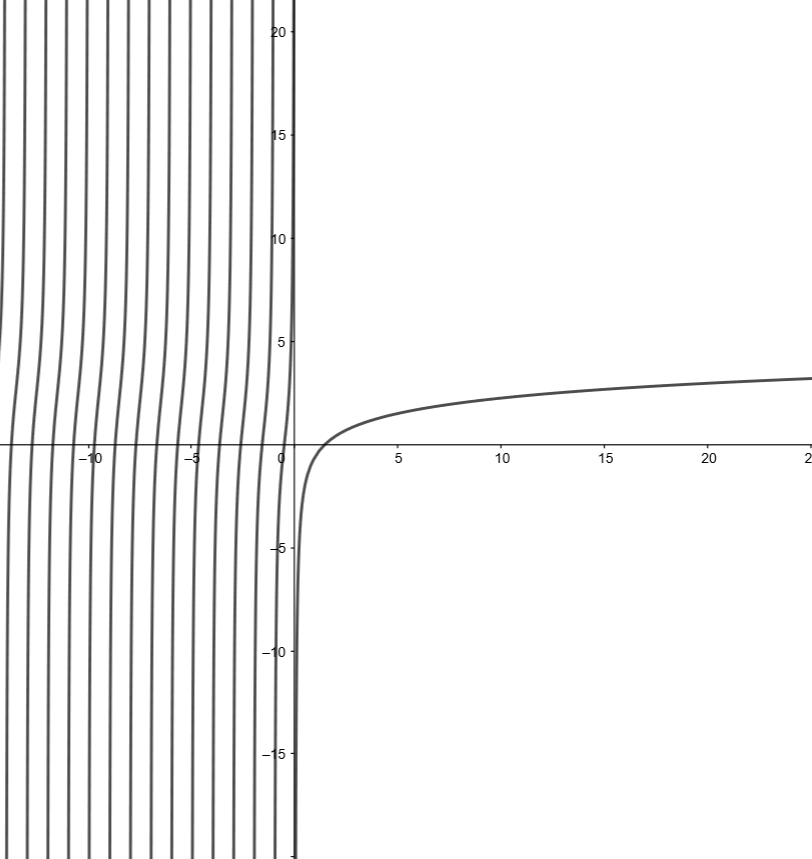}
	\caption{$\psi_0(x)$}
	\label{Figpsi}
\end{figure}
We notice that the incomplete gamma function $\gamma(a,z)$ is a Bernstein function and therefore we can use \eqref{defGammaInc} in order to define a new symbol for a subordinator (see \cite{beghingammainc}).
We now present some special functions which have been introduced in \cite[Section 18.3]{EHIII} and which will be useful further on. Let $\alpha\in(-\infty,+\infty),\beta \in (-1,+\infty)$, we define
\begin{align}
\label{defmu}
&\mu (x,\beta,\alpha):=\int_0^\infty \frac{x^{\alpha + y} y^\beta}{\Gamma(\beta +1) \Gamma(\alpha + y +1)} dy, \quad x>0.
\end{align}
The integral in \eqref{defmu}, focusing on complex variables, defines an analytical function in $x$, with branch-points in $x=0$ and $x=\infty$. It is an entire function of $\alpha$. For our convenience we also introduce the function
\begin{align}
\label{defnu}
\nu(x,\alpha):=\mu(x,0,\alpha) .
\end{align}
An further definition can be given as follows (\cite[formula (33.3)]{apelblat})
\begin{align}
\label{nugammainc}
\nu(x,\alpha)= e^x \int_\alpha^{\alpha+1} \frac{\gamma(y,x)}{\Gamma(y)} dy.
\end{align}
The Laplace transforms of these special functions can be easily obtained, in particular 
\begin{align}
\label{lapnu}
\int_0^\infty e^{-\lambda x} \nu(x,\alpha) dx&= \frac{1}{\lambda^{\alpha+1}\ln \lambda} \\
\label{lapmu}
\int_0^\infty e^{-\lambda x} \mu(x,\beta,\alpha) dx&= \frac{1}{\lambda^{\alpha+1}(\ln \lambda)^{\beta+1}} ,
\end{align}
with $\lambda >1$.
The functions $\mu$ and $\nu$ belongs to the class of higher transcendental functions as well as the famous Mittag-Leffler function. Such functions have been introduced in \cite{EHIII} together with a detailed discussion. However, in our view there is a misprint in \cite[page 221, formula (14)]{EHIII}. Thus, for the reader's convenience, we provide the following revised result concerning that formula.
\begin{lemma}
Let $\alpha \in \mathbb{R}$. The following holds true
	\begin{align}
	\label{serienumu}
	e^{-\alpha \vartheta} \nu (x e^\vartheta ,\alpha) =\sum_{k=0}^{\infty} \vartheta^k \mu (x,k,\alpha), \quad  x >0, \; \vartheta \in \mathbb{R}.
	\end{align}
\end{lemma}
\begin{proof}
Since
	\begin{align*}
	\displaystyle
	\nu(x e^\vartheta, \alpha)&= \int_0^\infty \frac{x^{\alpha + t}  (e^\vartheta)^{\alpha+t}}{\Gamma(\alpha +t+1)} dt =\int_0^\infty \frac{x^{\alpha + t}  e^{\vartheta \alpha}e^{\vartheta t}}{\Gamma(\alpha +t+1)} dt \\
	&=e^{\alpha \vartheta} \int_0^\infty \frac{x^{\alpha +t}\sum_{k=0}^{\infty}  \frac{(\vartheta t)^k}{\Gamma(k+1)}}{\Gamma(\alpha + t+ 1)} dt
	\\
	&=e^{\alpha \vartheta} \sum_{k=0}^{\infty} \vartheta^k \int_{0}^{\infty} \frac{x^{\alpha+t}t^k }{\Gamma(k+1) \Gamma(\alpha+t+1)} dt
	\\
	&=e^{\alpha \vartheta} \sum_{k=0}^{\infty} \vartheta^k \ \mu(x,k,\alpha)
	\end{align*}
we get the claimed result. A further check can be given by considering the Laplace transforms \eqref{lapnu} and \eqref{lapmu} in order to achieve the equality \eqref{serienumu}.
\end{proof}

\section{Governing equations}

Let us consider a continuous function $\varphi$ on $\mathbb{R}$ extended with zero on the negative part of the real line, that is $\varphi(x)=0$ if $x \leq 0$. From the Bochner subordination rule or in general, from the Phillips' representation (\cite{phillips}) we are able to obtain new operators through subordination. By following the same spirit we obtain a Marchaud (type) operator given by
\begin{align}
\label{genH}
-\Phi(-\partial_x) \varphi(x)=a\int_0^\infty (\varphi(x) - \varphi(x-y))\frac{e^{-by}}{y} dy.
\end{align}
Indeed, by taking into account \eqref{symbGamma}, we can immediately check that
\begin{align*}
\int_0^\infty e^{-\xi x} \Phi(-\partial_x) \varphi(x) \ dx = - \Phi(\xi)\tilde{\varphi}(\xi), \quad \xi >0
\end{align*}
where $\widetilde{\varphi}$ denote the Laplace transform of $\varphi$. We notice that, if $\varphi$ is $\alpha-$ Hölder, that is $|\varphi(x) - \varphi(x-y)|\leq M\, |y|^\alpha$ with $\alpha>0$ for some $M>0$, then 
\begin{align*}
| \Phi(-\partial_x) \varphi (x) | \leq a\,M\,\frac{\Gamma(\alpha)}{b^\alpha}.
\end{align*}
Notice that the Marchaud operator (of order $\alpha$) is well defined for (essentially) bounded and $\gamma$-H\"older functions with $\gamma>\alpha$. The readers can consult the interesting work  \cite{Ferrari} for further details. The operator \eqref{genH} can be also written as follows
\begin{align}
\label{RLH}
\mathcal{D}^\Phi_x \varphi (x) = \frac{d}{dx} \int_0^x \varphi(x-y)\, \overline{\Pi}(y)\, dy
\end{align}
which can be regarded as a Riemann-Liouville definition. From \eqref{lapTailLm} and the fact that $\varphi(0)=0$, we still get
\begin{align*}
\int_0^\infty e^{-\xi x} \mathcal{D}_x^\Phi \varphi(x) \ dx= \xi \tilde{\varphi}(\xi) \frac{\Phi(\xi)}{\xi}=\Phi(\xi)\tilde{\varphi}(\xi), \quad \xi >0.
\end{align*}
Since $\varphi(0)=0$, the Riemann-Liouville representation \eqref{RLH} is also equivalent to the Caputo-Dzherbashian (type) non-local operator
\begin{align}
\label{CapDef}
\mathfrak{D}^\Phi_x \varphi(x) = \int_0^x \varphi^\prime (x-y) \, \overline{\Pi}(y)\, dy
\end{align}
where $\varphi^\prime = d \varphi / dx$.  In the literature, the Caputo-Dzherbashian fractional derivative is well-known. The operator \eqref{CapDef} is similarly defined except for the convolution kernel. Caputo introduced his derivative in \cite{CapWork1, CapWork2} whereas, the second author actively investigated this operator starting from the papers  \cite{DzWork1, DzWork2}. The operator \eqref{CapDef} is well-defined for 
\begin{align*}
\varphi \in L^1(0, \infty) \quad \textrm{such that} \quad \varphi^\prime(x-y) \overline{\Pi}(y) \in L^1(0,x),\; \forall x.
\end{align*}
From the regularity of $\overline{\Pi}$, it turns out that \eqref{CapDef} is well-defined as $\varphi \in AC(0, \infty)$, that is the set of continuous functions on $(0, \infty)$ with $\varphi^\prime \in L^1(0, \infty)$. Moreover, this is confirmed from the Young's inequality for convolution from which
\begin{align}
\label{boundCD}
\|\mathfrak{D}^\Phi_x \varphi\|_1 \leq \| \varphi^\prime \|_1 \, \lim_{\lambda \downarrow 0} \frac{\Phi(\lambda)}{\lambda}
\end{align}
where, as a quick check shows, 
\begin{align}
\label{quickCheck}
\frac{1}{\lambda} a \ln \left( 1 + \frac{\lambda}{b} \right) \to \frac{a}{b} \quad \textrm{as} \quad \lambda \to 0.
\end{align}
As usual we denote by $\| \cdot \|_1$ the $L^1$-norm. Sometimes,  for the sake of clarity, we write $L^1(dx)$ or $L^1(dt)$ in place of $L^1(0, \infty)$ if some confusion may arise.

In view of \eqref{boundCD} we now introduce the spaces
\begin{align*}
W^{1,1}(0, \infty) = \{ \psi \in L^1(0, \infty):\;  \psi^\prime \in L^1(0, \infty) \}
\end{align*}
and
\begin{align*}
W^{1,1}_0 (0, \infty) = \{ \psi \in W^{1,1}(0, \infty):\, \psi(0)=0\}.
\end{align*}
Observe that $AC(I)$ coincides with $W^{1,1}(I)$ if $I$ is bounded. 

Before proceeding, let us remember that the Laplace transform is well defined for piecewise continuous functions of exponential order. In particular, let us consider $\varphi_1, \varphi_2$ on $(0, \infty)$ with Laplace transforms $\widetilde{\varphi}_1, \widetilde{\varphi}_2$. We recall that (Lerch's theorem) if $\widetilde{\varphi}_1=\widetilde{\varphi}_2$, then $\varphi_1=\varphi_2$ up to isolated points where at least one of the two functions is not continuous. Let $\mathcal{M}_\eta$ be the set of piecewise continuous functions of order $\eta\geq 0$, that is $|\varphi(z)| \leq C\, e^{\eta z}$ with $\widetilde{\varphi}$ defined on $(\eta, \infty)$. Further on we focus on the set of functions $\mathcal{M}_0 \cap C(0,\infty)$. 

We introduce the next result by first noticing that, from the representation \eqref{repHsing},  $h \nin W^{1,1}_0(0, \infty)$. This is because of the time dependent property introduced above in Section \ref{Sec:GammaSub}. Thus, the problem to address the right PDE for $h$ must be taken with some care. We now present the following result.

\begin{theorem} 
	\label{teoH}
	Let $v(t,x) \in C^{1,1}((0,\infty) , W^{1,1}_0(0,\infty);[0,\infty))$ be the solution to
	\begin{equation}
	\label{eqh}
	\left\lbrace	
	\begin{array}{ll}
	\displaystyle \frac{\partial}{\partial t} v(t,x) = - \mathcal{D}_x^\Phi v(t,x), &  t>0,\, x \in (0,\infty),\\
	\displaystyle v(0,x)=f(x), & f \in W_0^{1,1}(0, \infty),\\
	\displaystyle v(t,x)=0, & t>0,\, x \in (-\infty, 0]. 
	\end{array}	
	\right.	
	\end{equation}
	Then,
	\begin{align}
	\label{SOLv}
	v(t,x) = \int_0^x f(x-y)\, h(t,y)\, dy = \mathbf{E}_0[f(x-H_t), t < L_x].
	\end{align}
\end{theorem}

\begin{proof}
	Since $v (t,\cdot) \in W^{1,1}_0(0, \infty)$ for any $t>0$,
	\begin{align*}
	\int_0^\infty e^{-\xi x} \mathcal{D}^\Phi_x v(t,x)\, dx = \Phi(\xi)\, \widetilde{v}(t, \xi), \quad \xi >0
	\end{align*}
	where $\widetilde{v}(t, \xi)= \int_0^\infty e^{-\xi x} v(t,x)\, dx$, $\xi >0$. Let us write $\widetilde{v}(\lambda, x)= \int_0^\infty e^{-\lambda t} v(t,x)\, dt$, $\lambda >0$ and denote by $\widetilde{v}(\lambda, \xi)$ the double Laplace transform. The problem \eqref{eqh} takes the form
	\begin{align*}
	\lambda \tilde{v}(\lambda, \xi) - \tilde{f}(\xi) = - \Phi(\xi) \tilde{v}(\lambda, \xi)
	\end{align*}
	where, with obvious notation, $\widetilde{f}(\xi)$ is the Laplace transform of $f$.  Thus, 
	\begin{align*}
	\tilde{v}(\lambda, \xi)=\tilde{f}(\xi)\frac{1}{\lambda + \Phi(\xi)}
	= & \tilde{f}(\xi) \int_0^\infty e^{-\lambda t} e^{- t \Phi(\xi) }dt\\
	= & \tilde{f}(\xi) \mathbf{E}_0 \left[ \int_0^\infty e^{-\lambda t} e^{-\xi H_t} dt \right]
	\end{align*}
	where in the last step we have used \eqref{LapH}. From this,
	\begin{align*}
	\widetilde{v}(t,\xi) = \widetilde{f}(\xi) \, \mathbf{E}_0[e^{-\xi H_t}]
	\end{align*}
	and
	\begin{align*}
	v(t,x) = \int_0^x f(x-y)\, h(t,y)\, dy.
	\end{align*}
	Uniqueness follows from the Laplace transform technique (Lerch's theorem). Indeed, we are looking for a continuous solution, in particular $v\in C^{1,1}$. Thus, $\psi_t(\xi) = \widetilde{v}(t, \xi)$, $t>0$ and $\psi_x(\lambda) = \widetilde{v} (\lambda, x)$, $x>0$ are two Laplace transforms with unique inverses. We observe that 
\begin{align*}
\| v(t,\cdot)\|_1 \leq \| f\|_1 \| h(t,\cdot)\|_1
\end{align*}	
for any $t>0$, thus $e^{-\xi x} v(t,x)$ is obviously an element of $L^1(dx)$. Furthermore, we have that $\| v(t,\cdot)\|_\infty \leq \| f\|_\infty \| h(t,\cdot)\|_1$ for any $t>0$, thus $v(t,\cdot) \in \mathcal{M}_0 \cap C(0,\infty)$ as required. Indeed, $W^{1,1}_0(0,\infty)$ embeds into $L^\infty (0, \infty)$ and $f$ is essentially bounded. This holds only in the one dimensional case (for more details see section 11.2 of \cite{leoni2017first}). Similar arguments applies w.r. to the time variable by considering $\widetilde{h}(\lambda, x)$. In particular, 
	\begin{align*}
	\| v(\cdot,x) \|_1 \leq \int_0^x |f(x-y)| \kappa(y)\, dy = (J |f|)(x)
	\end{align*}
	where $J$ is a non-local (sometimes termed fractional) integral. This integral can be compared with the solution of \eqref{eqw} below. Here, the function $t \mapsto v(t,x)$ is continuous and integrable for any $x$ and this allows us to proceed with the Laplace technique w.r. to the time variable. We can also observe that $t \to h(t,x)$ is continuous and, since $\Gamma(at)$ is asymptotically faster than $b^{at}x^{at-1}$,
	\begin{align*}
	\lim_{t \to 0} h(t,x)=0 \quad \text{and }\quad \lim_{t\to \infty} h(t,x)=0,
	\end{align*}
	from which we conclude that $h(\cdot,x)$ is bounded for any $x>0$. Proceeding as before, we obtain $\| v(\cdot,x)\|_\infty \leq \| f\|_1 \| h(\cdot,x)\|_\infty$, thus $v(\cdot,x) \in \mathcal{M}_0 \cap C(0,\infty)$ as required.\\
	
	We obtain the probabilistic representation from the fact that
	\begin{align*}
	v(t,x) = \int_0^x f(x-y)\, h(t,y)\, dy = \int_0^\infty f(x-y)\, h(t,y)\, \mathbf{1}_{(y < x )} dy,
	\end{align*}
	that is
	\begin{align*}
	v(t,x) = \mathbf{E}_0[f(x-H_t) \mathbf{1}_{(H_t < x)}].
	\end{align*}
	From \eqref{inverseRelation}, we have that $\mathbf{1}_{(H_t < x)}$ is equivalent to $\mathbf{1}_{(t < L_x)}$ under $\mathbf{E}_0$. \\

	-) We check that $v(t, \cdot) \in C^{1}(0, \infty)$ for any fixed $t>0$. First we observe that $v(t, \cdot) \in C(0,\infty)$ for any $t>0$. Moreover, $\forall\, t>0$, 
	\begin{align}
	\label{convder}
	\frac{d}{dx} v(t,x)&= \frac{d}{dx} \int_{0}^{x} f(x-y) h(t,y) dy \notag \\
	&=f(0) h(t,x) + \int_{0}^{x} f^\prime (x-y) h(t,y) dy \notag\\
	&=\int_{0}^{x} f^\prime (x-y) h(t,y) dy.
	\end{align}
	Since $f \in W_0^{1,1}(0, \infty)$ and $h(t,\cdot)\in L^1(dx)$ $\forall\, t>0$, we conclude that $\frac{d}{dx} v(t,\cdot) \in C(0, \infty)$ for any $t>0$. From the Young's inequality, we notice that the convolution in \eqref{convder} is in $L^1(dx)$ and therefore $v(t,\cdot) \in W_0^{1,1}(0,\infty), \ t>0$.\\

	-) We check that $v(\cdot, x) \in C^{1}(0, \infty)$ for any fixed $x>0$. The derivative
	\begin{align*}
	\frac{d}{dt} v(t,x)= \frac{d}{dt} \int_0^x f(x-y) h(t,y) dy
	\end{align*}
	can be written in terms of
	\begin{align}
	\label{derivatah}
	\frac{d}{dt} h(t,x)&=\frac{d}{dt} \left[\frac{b^{at} x^{at-1}}{\Gamma(at)} e^{-bx}\right]= \frac{a b^{at} x^{at-1}}{\Gamma(at)} \left[\ln x +\ln b-\psi_0(at)\right] e^{-bx} \notag\\
	&=a h(t,x) \left[\ln x +\ln b-\psi_0(at)\right].
	\end{align}
	%
	The idea is to proof that $\frac{d}{dt} h(t,x)$ is continuous and bounded in $t$ for each $x>0$, then it is integrable. First we show that
	\begin{align}
	\label{limhpsi}
	h(t,x) \psi_0(at)=-\frac{e^{-bx}}{x} \quad as\ t \to 0, \quad h(t,x) \psi_0(at)=0 \  \quad as\ t \to \infty \ .
	\end{align}
	Let us recall (\cite[ formula (13.1.5.1)]{krantz})
	\begin{align}
	\label{gammalim}
	\Gamma(z)=\lim_{n \to \infty} \frac{n!n^z}{z(z+1)\cdots (z+n)}.
	\end{align}
	Since $\psi_0(z)=\frac{\Gamma^\prime(z)}{\Gamma(z)}$, then
	\begin{align}
	\label{psilim}
	\psi_0(z)=\lim_{n \to \infty} \log n - \left(\frac{1}{z} + \frac{1}{z+1} + \cdots +\frac{1}{z+n}\right) .
	\end{align}
From \eqref{gammalim} and \eqref{psilim} we write
	\begin{align*}
	\frac{\psi_0(z)}{\Gamma(z)}&=\lim_{n \to \infty} \frac{\log n - \left(\frac{1}{z} + \frac{1}{z+1} + \cdots +\frac{1}{z+n}\right)}{\frac{n!n^z}{z(z+1)\cdots (z+n)}}\\
	&=\lim_{n \to \infty} \frac{z\log n - \left(1 + \frac{z}{z+1} + \cdots +\frac{z}{z+n}\right)}{\frac{n!n^z}{(z+1)\cdots (z+n)}}  .
	\end{align*}
	As $z \to 0$, we obtain that
	\begin{align}
	\label{psidivgammazero}
	\frac{\psi_0(0)}{\Gamma(0)}=\lim_{n \to \infty} \frac{0 - \left(1 + 0 + \cdots +0\right)}{\frac{n!}{n!}}=-1 
	\end{align}
	whereas, due to the asymptotic contribution of $n^z$,
	\begin{align}
	\label{psidivgammainf}
	\lim_{z \to \infty} \frac{\psi_0(z)}{\Gamma(z)}=0.
	\end{align}
	From \eqref{psidivgammazero} and \eqref{psidivgammainf} we respectively obtain
	\begin{align*}
	\lim_{t \to 0} h(t,x) \psi_0(at)=\lim_{t \to 0} { b^{at} x^{at-1}e^{-bx}} \frac{\psi_0(at)}{\Gamma(at)}=-\frac{e^{-bx}}{x} ,
	\end{align*}
	and
	\begin{align*}
	\lim_{t \to \infty} h(t,x) \psi_0(at)=\lim_{t \to \infty} { b^{at} x^{at-1}e^{-bx}} \frac{\psi_0(at)}{\Gamma(at)}=0, 
	\end{align*}
	 which prove  \eqref{limhpsi}.
	Summing up, for each fixed $x>0$, from \eqref{derivatah}, $h(t,x) (\ln x + \ln b)$ is continuous in $t$ and, from \eqref{limhpsi}, we have that $h(t,x) \psi_0(at)$ is continuous $\forall x \ne 0$, hence $\frac{d}{dt} h$ is continuous and bounded, so it is in $ L^1(dt)$. This guarantees that $\frac{d}{dt} v(\cdot,x)$ is $C(0,\infty)$ for any $x>0$ 
	and concludes the proof.
\end{proof}

Notice that $L_x$ in \eqref{SOLv} can be regarded as the stopping time depending on the threshold $x>0$. Indeed, $L$ is an exit time for $H$.\\

Let us consider once again the representation \eqref{repHsing} of $h(t,x)$. The profile of \eqref{repHsing} is given in Figure \ref{FigProfilehstar}. As we can see, $h$ does not seem to be in the domain of \eqref{genH}. Indeed, for $at<1$, $h$ is not continuous at $x=0$ (so it can not be of the Hölder class). In general, for a positive integer $n$, if $n<at\leq n+1$, then $h(t,x)$ is of class $C^{n-1}[0, \infty) \cap C^n(0, \infty)$, thus, $h \notin C^n[0, \infty)$. Such a property is well described in Figure \ref{FigProfilehstar}. 
\begin{figure}[h]
	\centering
	\includegraphics[width=7cm]{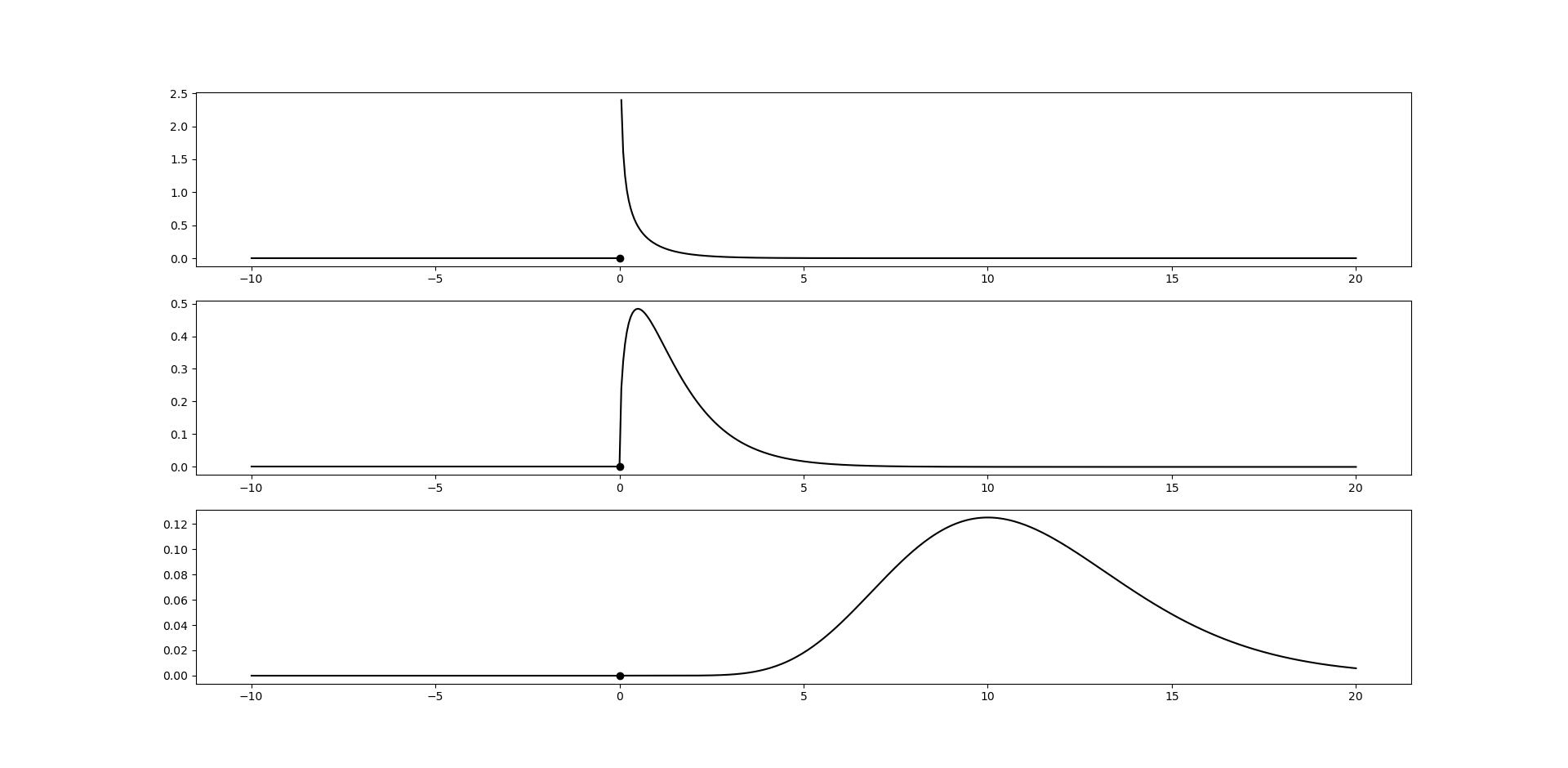}\\
	\caption{The profile of $h(t,x)$ with $a=b=1$ ; above the case $t=0.5$ with discontinuity at $x=0$, in the middle the case $t=1.5$ with a non-differentiable point in $x=0$, below the case $t=11$. }
	\label{FigProfilehstar}
\end{figure}
\\

%

For completeness we present an analogue result of Theorem \ref{teoH}. The following result is concerned with the non-local operator in time and therefore, the probabilistic representation of the solution involves the inverse Gamma subordinator.

\begin{theorem}
	Let $r(t,x) \in C^{1,1} ( AC(0,\infty), (0,\infty); [0,\infty) )$ be the solution to
	\begin{align*}
	\left\lbrace	
	\begin{array}{ll}
	\displaystyle \mathfrak{D}_t^\Phi r(t,x) = - \frac{\partial}{\partial x} r(t,x), &  t>0,\, x \in (0,\infty),\\
	\displaystyle r(0,x)=f(x), & f \in C_b(0, \infty),\\
	\displaystyle r(t,x)=0, & t>0,\, x \in (-\infty, 0]. 
	\end{array}	
	\right.	
	\end{align*}
	Then
	\begin{align*}
	r(t,x)=\mathbf{E}_0[f(x-L_t), t<H_x].
	\end{align*}
\end{theorem}
\begin{proof}
	The proof follows after some adaptation of the proof of the previous theorem. We have only to add the fact that $\mathfrak{D}^\Phi_t$ is a convolution operator on the set $AC(0, \infty)$ and therefore 
	\begin{align*}
	\int_0^\infty e^{-\lambda t} \mathfrak{D}^\Phi_t r(t,x)\, dt 
	= & \left( \int_0^\infty e^{-\lambda t} r^\prime(t,x)\, dt \right) \left( \int_0^\infty e^{-\lambda t} \overline{\Pi}(t)\, dt \right)\\
	= & \left( \lambda \widetilde{r}(\lambda, x) - r(0,x) \right) \frac{\Phi(\lambda)}{\lambda}
	\end{align*}
	where we used \eqref{lapTailLm}. Then, by using Laplace transforms techniques, we have the claim. Concerning the Laplace machinery, we remark that $\| r(t, \cdot)\|_\infty \leq \| f \|_\infty $ for any $t>0$. Moreover,  for any $x \in (0, \infty)$,
	\begin{align*}
	\| r(\cdot, x) \|_1 \leq \| l(\cdot, x)\|_1 \, \int_0^x |f(y)|dy 
	\end{align*}
	where $\| l(\cdot, x)\|_1 = \lim_{\lambda \to 0^+} \Phi(\lambda)$ given in \eqref{quickCheck}.
\end{proof}

\begin{remark}
	We only recall that $\mathfrak{D}_t^\Phi r = \mathcal{D}_t^\Phi (r-r_0)$ where $\mathcal{D}^\Phi_t$ has been defined in \eqref{RLH}. Thus the previous problem can be written in terms of the Riemann-Liouville (type) derivative in place of the Caputo (type) derivative.
\end{remark}

\begin{remark}
	We remark that the class of the initial datum in Theorem \ref{teoH} is larger than $W^{1,1}_0(0, \infty)$. Let us consider $v(0,x)=\mathbf{1}_{(0, \infty)}(x)$ for instance. Then, we still have a solution. In particular,  
	\begin{align*}
	v(t,x)=\int_{0}^x h(t,y) dy = \mathbf{P}_0(H_t < x).
	\end{align*}
	This comes out also from the probabilistic representation \eqref{SOLv}. Indeed, 
	\begin{align*}
	v(t,x) = \mathbf{P}_0(t < L_x) = \mathbf{P}_0(H_t < x)
	\end{align*}
	where the last identity follows from the relation \eqref{inverseRelation}.
\end{remark}

\section{Main results}

\label{sec:MainResults}

\subsection{Densities and kernels}
As for $\gamma(a,z)$, let us define the incomplete digamma function $\psi_0(a,z)$	
through the integral representation 
\begin{align*}
\Psi_0(a,z):=\frac{1}{\gamma(a,z)} \int_0^z y^{a-1} e^{-y} \ln y \ dy \quad a>0, z \geq 0.
\end{align*}
We observe that
\begin{align*}
\lim_{z \to \infty} \Psi_0(a,z)=\psi_0(a) .
\end{align*}

We provide the following representation result.
\begin{theorem}
	For $a,b>0$  the following representation holds
	\begin{align}
	\label{repl}
	l(t,x)
	=\frac{a \ \gamma(ax,bt)}{\Gamma(ax)} \left(\psi_0(ax) -\Psi_0(ax,bt) \right) \quad t>0, \; x>0 .
	\end{align}
\end{theorem}
\begin{proof}
For positive $x$, we have, agreeing with \cite{kumarPS}
\begin{align}
\mathbf{P}_0(L_t > x) 
& =\mathbf{P}_0(H_x < t) \notag\\
&=\int_0^t \frac{b^{ax} }{\Gamma(ax)} y^{ax-1} e^{-by} dy \notag \\
&=\int_0^{bt} \frac{b^{ax} }{\Gamma(ax)} \left(\frac{z}{b}\right)^{ax-1} e^{-z} \frac{dz}{b} \notag \\
&=\int_0^{bt} \frac{1}{\Gamma(ax)} z^{ax-1} e^{-z} dz \notag \\
\label{ripLt}
&=\frac{\gamma(ax,bt)}{\Gamma(ax)}
\end{align}
where the incomplete gamma function \eqref{defGammaInc} is involved. Thus, from the relations 
\begin{align*}
\frac{d}{dx} \gamma(ax, bt) = \frac{d}{dx} \int_0^{bt}y^{ax-1 } e^{-y}dy=\int_0^{bt}y^{ax-1 } e^{-y} \ \ln(y) \ dy,
\end{align*}
and from (formula 6.3.1 in \cite{handbook})
\begin{align}
\label{derGammaFunc}
\frac{d}{dx} \Gamma(x) = \psi_0(x) \Gamma(x),
\end{align}
we get
\begin{align*}
l(t,x)&=-\frac{d}{dx}\mathbf{P}_0(L_t>x)=-\frac{d}{dx}\frac{\gamma(ax,bt)}{\Gamma(ax)}\\
&=\frac{a \ \gamma(ax,bt) \ \Gamma(ax) \ \psi_0(ax)-\Gamma(ax)  \frac{d}{dx} \int_0^{bt}y^{ax-1 } e^{-y}dy} {\Gamma(ax)^2}\\
&=\frac{a \ \gamma(ax,bt) \ \psi_0(ax) -a\int_0^{bt}y^{ax-1 } e^{-y} \ \ln(y) \ dy}{\Gamma(ax)}.
\end{align*}
From \eqref{derGammaFunc} and \eqref{defdigamma}, we write
\begin{align*}
\int_0^{bt}y^{ax-1 } e^{-y} \ \ln(y) \ dy = \Psi_0(ax,bt) \gamma(ax,bt)
\end{align*}
and therefore, the claim follows.\\

We now proceed by considering the Laplace technique. We observe that (\cite[formula (3.13)]{triangarray})
	\begin{align}
	\label{lapL}
	\int_0^\infty e^{-\lambda t} l(t,x) dt= \frac{\Phi(\lambda)}{\lambda} e^{-x \Phi(\lambda)}, \quad \lambda > 0
	\end{align}
is verified by the representation \eqref{repl} as expected. Since
\begin{align*}
\int_0^\infty e^{-\lambda t} \gamma(ax,bt)dt&=\int_0^\infty e^{-\lambda t} \int_0^{bt} y^{ax-1} e^{-y}   dy\ dt\\
&=\frac{1}{\lambda}  \int_0^\infty e^{-\lambda \frac{y}{b}} e^{-y} y^{ax-1}  dy \\
&=\frac{1}{\lambda} \Gamma(ax) \left(1+\frac{\lambda}{b}\right)^{-ax}\\
&=\frac{\Gamma(ax)}{\lambda} e^{-x \Phi(\lambda)} 
\end{align*}
we write
\begin{align*}
\int_0^\infty e^{-\lambda t} a \frac{\gamma(ax,bt)}{\Gamma(ax)} \psi_0(ax) dt=a \psi_0(ax) \frac{1}{\lambda} e^{-x \Phi(\lambda)}.
\end{align*}
Moreover, 
\begin{align*}
\int_0^\infty e^{-\lambda t} \int_0^{bt} y^{ax-1} e^{-y} \ln y \ dy \ dt&= \frac{1}{\lambda} \int_{0}^{\infty} e^{-\frac{\lambda}{b} y} y^{ax-1} e^{-y} \ln y \ dy\\
&=\frac{1}{\lambda} \frac{1}{\left(1+\frac{\lambda}{b}\right)^{ax}} \Gamma(ax) \left[\psi_0(ax) -\ln\left(1+\frac{\lambda}{b} \right) \right]\\
&=\frac{1}{\lambda} \Gamma(ax) \left[\psi_0(ax) -\ln\left(1+\frac{\lambda}{b} \right) \right] e^{-x \Phi(\lambda)},
\end{align*}
where the second-last equality comes from  (\cite[formula 4.352]{table})
\begin{align*}
\int_0^\infty x^{c-1} e^{-dx} \ln x dx =\frac{1}{d^c} \Gamma(c)[\psi_0(c) - \ln d] \quad c,d>0,
\end{align*}
and therefore
\begin{align*}
\int_0^\infty e^{-\lambda t} \frac{a}{\Gamma(ax)} \int_0^{bt} y^{ax-1} e^{-y} \ln y \ dy \ dt &=\frac{a}{\lambda} \left[\psi_0(ax) -\ln\left(1+\frac{\lambda}{b} \right) \right] e^{-x \Phi(\lambda)} \\
&=\frac{1}{\lambda} \left[a\psi_0(ax) -\Phi(\lambda) \right] e^{-x \Phi(\lambda)}.
\end{align*}
By collecting all together the previous Laplace transforms, 
\begin{align*}
\int_0^\infty e^{-\lambda t} l(t,x) dt
= & a \psi_0(ax) \frac{1}{\lambda} e^{-x \Phi(\lambda)} - \frac{1}{\lambda} \left[a\psi_0(ax) -\Phi(\lambda) \right] e^{-x \Phi(\lambda)}\\
=& \frac{\Phi(\lambda)}{\lambda} e^{-x \Phi(\lambda)}, \quad \lambda>0.
\end{align*}
This concludes the proof.
\end{proof}

%
%

From the $\lambda$-potential
\begin{align*}
\mathbf{E}_0 \left[ \int_0^\infty e^{-\lambda H_t}\, dt \right] = \frac{1}{a \ln \left( 1+ \frac{\lambda}{b} \right)} = \int_0^\infty e^{-\lambda x} \kappa(x)\, dx ,\quad \lambda >0
\end{align*}
we can write the potential density
\begin{align*}
\kappa(x) = \int_0^\infty h(t,x)\, dt.
\end{align*}
On the other hand 
\begin{align*}
\lim_{x \downarrow 0} \int_0^\infty e^{-\lambda t} l(t,x)\, dt = \frac{1}{\lambda} a \ln \left( 1 + \frac{\lambda}{b} \right) 
\end{align*}
that is, $l(t,0)= \overline{\Pi}(t)$ as introduced in \eqref{lapTailLm}. Let us write 
$$\ell(t)=l(t,0).$$ 
We have that
\begin{align}
\label{convSonine}
\int_0^1 \kappa(z)\, \ell(1-z)\, dz =1
\end{align}
and therefore, $\kappa$ and $\ell$ are associated Sonine kernels (see \cite{sonine}, \cite{samko}). \\
Next we rewrite $\kappa$ and $\ell$ using some  information on the Gamma subordinator.

Let us first introduce the exponential integral
\begin{align*}
E_1(x):=\int_{x}^\infty \frac{e^{-z}}{z} dz, \quad x>0.
\end{align*}

We now present the following result.

\begin{theorem}
\label{thm:Sonine}
The associated Sonine kernels $\kappa$ and $\ell$ are given by
\begin{align}
\label{SonineKernels}
\kappa(x) = \frac{b}{a} e^{-bx} \nu(bx, -1), \quad \ell(x) = a E_1(bx), \quad x>0.
\end{align}
\end{theorem}

\begin{proof}
Observe that
\begin{align}
\label{tailE1}
\overline{\Pi}(t)=a\int_t^\infty \frac{e^{-by}}{y} dy = a\int_{bt}^\infty \frac{e^{-z}}{z} dz = a E_1(bt).
\end{align}
For $t>0$, $x>0$,
	\begin{align}
	l(t,x) 
	= & \int_0^t h(x,t-s) \overline{\Pi}(s) ds \label{convhPi} \\
	= & \int_0^t \frac{b^{ax}}{\Gamma(ax)} (t-s)^{ax-1} e^{-b(t-s)} a E_1(bs) \ ds \label{repSonineL}
	\end{align}
where the formula \eqref{repSonineL} can be easily obtained from the definition \eqref{tailE1} of $E_1$. Then we focus on \eqref{convhPi}. From \eqref{LapH} we have that
\begin{align*}
	\int_0^\infty e^{-\lambda s} \frac{b^{ax}}{\Gamma(ax)} s^{ax-1} e^{-bs } ds=
e^{-ax \ln\left(1+\frac{\lambda}{b}\right)}.
	\end{align*}
From this, by taking into account \eqref{lapTailLm} we get the Laplace transform 
	\begin{align*}
	\int_0^\infty e^{-\lambda t} \int_0^t h(x,t-s) \overline{\Pi}(s) ds\, dt =\frac{\Phi(\lambda)}{\lambda} e^{-x \Phi(\lambda)}, \quad \lambda>0
	\end{align*}
which coincides with \eqref{lapL}. This proves that \eqref{convhPi} holds true.	\\
	
	Now we prove \eqref{SonineKernels}. Since
	\begin{align*}
	\int_0^\infty l(t,x) dx=1,
	\end{align*}
from the previous result we have that
	\begin{align*}
	1
	= & \int_0^\infty \left(\int_0^t h(x,t-s) \overline{\Pi}(s) ds\right) dx \\
	= & \int_0^t \overline{\Pi}(s) e^{-b(t-s)} \left(\int_0^\infty \frac{b^{ax}}{\Gamma(ax)} (t-s)^{ax-1} dx\right) ds \\
	= & \int_0^t \overline{\Pi}(s) e^{-b(t-s)} \frac{b}{a} \left(\int_0^\infty \frac{[b(t-s)]^{y-1}}{\Gamma(y)} dy\right) ds.	
\end{align*}
From \eqref{defnu} we get that
\begin{align*}
	1 = & \int_0^t aE_1(bs) e^{-b(t-s)} \frac{b}{a} \nu(b(t-s), -1) ds.
	\end{align*}
	This concludes the proof.
\end{proof}

We now move to the elliptic problem associated with an Abel (type) equation. Let $f \in C_b(0, \infty)$ be such that
\begin{align*}
f(y) \kappa(x-y) \in L^1(0,x),\, \forall x.
\end{align*}
The unique continuous solution $w$ to
\begin{equation}
\label{eqw}
\left\lbrace
\begin{array}{ll}
\displaystyle f(x)=\mathcal{D}_x^\Phi w(x) & x \in (0,\infty),\\
\displaystyle w(x)=0 &  x \in (-\infty, 0].
\end{array}
\right .
\end{equation}
is given by
\begin{align*}
w(x) = \int_0^x f(x-y) \kappa(y)\, dy = \mathbf{E}_0\left[ \int_0^{L_x} f(x-H_t)\, dt \right]
\end{align*}
We immediately see that 
\begin{align*}
w(x) = \int_0^\infty v(t,x)\, dt
\end{align*}
where $v$ is the solution to \eqref{eqh}. Since 
		\begin{align*}
		\lim_{t\to \infty} v(t,x)=0
		\end{align*}
and $v(0,x)=f(x)$, the problem \eqref{eqh} takes the form \eqref{eqw} just integrating w.r. to the time variable. Concerning the solution $w$ we have that, by definition, $\kappa(x)=\int_0^\infty h(t,x) dt$. The probabilistic representation can be directly obtained from \eqref{SOLv}.

		

\subsection{Real moments}

We provide here a formula for the real $q$-th moment of the inverse process $L_t$. In the literature only the first two moments are known. The problem arises on the computation of the inverse Laplace transform
\begin{align}
\int_0^\infty e^{-\lambda t} \mathbf{E}[(L_t)^q]\, dt 
= &  \frac{\Phi(\lambda)}{\lambda} \int_0^\infty x^q e^{-x \Phi(\lambda)} dx\notag \\ 
= & \frac{\Gamma(q+1)}{\lambda\, (\Phi(\lambda))^q}, \quad \lambda>0, \; q>0.
\label{lapMomL}
\end{align}
In the last evaluation we have taken into account \eqref{lapL}. As we can immediately see the formulas \eqref{lapMomL} and \eqref{lapmu} are evidently related. This inspires the forthcoming discussion.

Further on the following notation $f \sim g$ stands for $f(z)/g(z) \to c$ as $z\to \infty$ where $c$ is a positive constant.

\begin{theorem}
	\label{teomomenti}
Let $q \in [1,\infty)$, $t>0$.
	\begin{itemize}
		\item[(i)] The moments of the inverse process $L$ are given by
	\begin{align}
	\label{momenti2}
	\mathbf{E}_0[L_t^q]&=\frac{\Gamma(q+1)}{a^q} \sum_{n=0}^\infty e^{-bt} \mu(bt,q-1,n)  \\
	\label{momenti_int}
	&=\frac{b\ \Gamma(q+1)}{a^q} \int_0^t   e^{-bs} \mu(bs,q-1,-1) \ ds .
	\end{align} 
	\item[(ii)] Moreover, as $t\to \infty$,
	\begin{align}
	\label{asintmom}
			\mathbf{E}_0[L_t^q] \sim \left(\frac{b \ t}{a} \right)^q .
	\end{align}
\end{itemize}
\end{theorem}
\begin{remark}
The result in point $ii)$ has been introduced in \cite{kumarPS} for $b=1$. We confirm such a result in the general case  $b \ne 1$. In particular, we obtain \eqref{asintmom} by considering an alternative proof via our arguments.
For this reason we decided to state \eqref{asintmom} in Theorem \ref{teomomenti}.
\end{remark}
\begin{proof}
We present the proof of Theorem \ref{teomomenti} as follows:
	\begin{itemize}
		\item[(i)]
	If $H_t$ is a subordinator with symbol $\Phi (\lambda)$ and $L_t$ is an inverse subordinator to $H_t$, then by  proceeding as in \cite{veillette}, the time-Laplace transform of the moments $\mathbf{E}_0[L_t^q]$ can be easily given as
	\begin{align*}
	\mathbf{E}_0\left[\int_0^\infty e^{-\lambda t} (L_t)^q \ dt \right]&=\int_0^\infty e^{-\lambda t} \int_0^\infty x^q l(t,x) dx \ dt\\
	&=\int_0^\infty x^q \frac{\Phi(\lambda)}{\lambda} e^{-x \Phi(\lambda)} dx \\
	&=\frac{\Phi(\lambda)}{\lambda} \frac{\Gamma(q+1)}{\Phi(\lambda)^{q+1}}\\
	&=\frac{\Gamma(q+1)}{\lambda \Phi(\lambda)^q}
	=\frac{\Gamma(q+1)}{a^q \lambda \left[\ln(1+\frac{\lambda}{b})\right]^q}, \quad \lambda >0 .
	\end{align*}
	On the other hand, 
	\begin{align*}
	& \int_0^\infty e^{-\lambda t} \left[\frac{\Gamma(q+1)}{a^q} \sum_{n=0}^\infty e^{-bt} \mu(bt,q-1,n)\right] dt\\
	&= \frac{\Gamma(q+1)}{a^q} \sum_{n=0}^\infty \int_0^\infty e^{-(\lambda+b) t} \mu(bt,q-1,n) dt \\
	&=\frac{\Gamma(q+1)}{a^q} \sum_{n=0}^\infty\frac{1}{b} \int_0^\infty e^{-\frac{\lambda+b}{b} t} \mu(t,q-1,n) dt \quad \lambda > 0.
	\end{align*}
	Now, by using \eqref{lapmu},
	\begin{align*}
& \frac{\Gamma(q+1)}{a^q} \sum_{n=0}^\infty\frac{1}{b} \int_0^\infty e^{-\frac{\lambda+b}{b} t} \mu(t,q-1,n) dt\\
&=\frac{\Gamma(q+1)}{a^q} \sum_{n=0}^\infty\frac{1}{b} \frac{1}{\left(1+\frac{\lambda}{b}\right)^{n+1}} \frac{1}{\left[\ln\left(1+\frac{\lambda}{b}\right)\right]^q}\\
&=\frac{\Gamma(q+1)}{a^q} \frac{1}{\left[\ln\left(1+\frac{\lambda}{b}\right)\right]^q} \frac{1}{b} \frac{b}{\lambda}\\
&=\frac{\Gamma(q+1)}{a^q \lambda \left[\ln(1+\frac{\lambda}{b})\right]^q}.
	\end{align*}
	and this proves the identity.
	
	 We now  prove \eqref{momenti_int}.	From (\cite[formula (13) in section 18.3]{EHIII}) we know
	 \begin{align}
	 \label{derivatemu}
	 \frac{d^n}{dx^n} \mu(x,\beta,\alpha)=\mu(x,\beta,\alpha-n) \ .
	 \end{align}	
	 Then, by integration by parts
	 \begin{align*}
	 & \int_0^t   e^{-bs} \mu(bs,q-1,-1) \ ds \\
	 &=\int_0^t   e^{-bs} \frac{d}{ds} \frac{1}{b}\mu(bs,q-1,0) \ ds \\
	 &=\frac{e^{-bs}}{b} \mu(bs,q-1,0) \big\vert_0^t + \int_0^t   e^{-bs} \mu(bs,q-1,0) \ ds\\
	 &= \frac{e^{-bt}}{b} \mu(bt,q-1,0) + \frac{e^{-bs}}{b} \mu(bs,q-1,1) \big\vert_0^t + \int_0^t   {e^{-bs}} \mu(bs,q-1,1) \ ds\\
	 &= \dots =\sum_{n=0}^\infty e^{-bt} \frac{\mu(bt,k-1,n)}{b}  .
	 \end{align*}
	From this we get 
	 \begin{align}
	 \label{momLtseriemu}
	 \mathbf{E}_0[L_t^q] 
	 = & \frac{\Gamma(q+1)}{a^q} \sum_{n=0}^\infty e^{-bt} {\mu(bt,q-1,n)}\\
	 = & \frac{b\ \Gamma(q+1)}{a^q} \int_0^t   e^{-bs} \mu(bs,q-1,-1) \ ds \ .
	 \end{align}
	 As $t \to 0$, both $\nu(t,-1)$ and $\mu(t,\beta,-1)$ diverge. We now analyze this singularity. First we can observe that for any $\varepsilon >0$ and $q \in [1,\infty)$,
	 \begin{align*}
	 \int_\varepsilon^t e^{-bs} \mu(bs,q-1,-1) ds < \infty \ .
	 \end{align*}
	 Then, we only consider the integral in \eqref{momenti_int} near the origin. We use the asymptotics (\cite{apelblat}, pages 178-179),
	 \begin{align}
	 \label{nu0}
	 \nu(t,-1) &\sim \frac{1}{t(\ln t)^2}, \quad t \to 0,\\
	 \label{mu0}
	 \mu(t, \beta,-1) &\sim \frac{(1- \xi)^\beta t^{-\xi}}{\Gamma(\beta+1) \Gamma(1-\xi)}, \quad 0<\xi<1, \ t \to 0.
	 \end{align}
	 For $\varepsilon >0$ small enough (in particular $\varepsilon < \frac{1}{b}$), by using \eqref{nu0}
	 \begin{align*}
	 \int_0^\varepsilon e^{-bs} \nu(bs,-1) \ ds \sim \int_0^\varepsilon \frac{e^{-bs}}{bs (\ln bs)^2} ds <\infty,
	 \end{align*}
	 from which we deduce convergence. 
	 
	 Furthermore, from the mean value theorem for integrals we know that $\exists c_\varepsilon \in (0,\varepsilon) $ such that
	 \begin{align*}
	 \int_0^\varepsilon e^{-bs} \nu(bs,-1) \ ds = \varepsilon e^{-b c_\varepsilon} \nu(b c_\varepsilon,-1) \ ,
	 \end{align*}
	 from which, by taking into account formula \eqref{nu0},
	 \begin{align*}
	 \lim_{\varepsilon \to 0} \int_0^\varepsilon e^{-bs} \nu(bs,-1) \ ds =\lim_{\varepsilon \to 0} \frac{e^{-b \varepsilon}}{(\ln \varepsilon)^2} = 0 \ .
	 \end{align*}
	 In conclusion, we have loss mass near zero and this confirms that  
	 \begin{align*}
	 \int_0^\varepsilon e^{-bs} \nu(bs,-1) \ ds < \infty \ .
	 \end{align*}
	 Thus, also by the previous argument we obtain convergence of the integral in \eqref{momenti_int} with $q=1$.\\
	 We can use the same argument for $q > 1$ through $\eqref{mu0}$.

	\item[(ii)]
	In \cite{kumarPS} the authors have considered the following proof: recall that 
	\begin{align*}
	\mathbf{E}_0\left[\int_0^\infty e^{-\lambda t} (L_t)^q \ dt \right]=\frac{\Gamma(q+1)}{a^k \lambda \left[\ln(1+\frac{\lambda}{b})\right]^q} ;
	\end{align*}
	since
	\begin{align*}
	\ln\left(1+\frac{\lambda}{b}\right) \sim \frac{\lambda}{b} \quad \text{as } \lambda \to 0,
	\end{align*}
	then
	\begin{align*}
	\mathbf{E}_0\left[\int_0^\infty e^{-\lambda t} (L_t)^q \ dt \right]\sim \frac{b^q}{a^q} \frac{\Gamma(q+1)}{\lambda^{q+1}} \quad \text{as } \lambda \to 0 .
	\end{align*}
	By Tauberian theorem (see e.g. \cite{levy_bertoin}) they conclude
	\begin{align*}
	\mathbf{E}_0[L_t^q] \sim \left(\frac{b \ t}{a} \right)^q \ \ \text{as} \ t \to \infty \ , \ q>0.
	\end{align*}
	We alternatively prove \eqref{asintmom} by considering our representation \eqref{momenti_int}. First we observe that
					\begin{align}
					\mathbf{E}_0[L_t^q]&=
					\frac{b\ \Gamma(q+1)}{a^q} \int_0^t   e^{-bs} \mu(bs,q-1,-1) \ ds \notag\\
					&=\frac{b\ \Gamma(q+1)}{a^q} \int_0^\infty \int_0^t   e^{-bs} \frac{(bs)^{ y-1} y^{q-1}}{\Gamma(q) \Gamma(y)} dy \ ds \notag\\
					&=\frac{q}{a^q} \int_0^\infty \int_0^{bt}   e^{-z} \frac{z^{ y-1} y^{q-1}}{ \Gamma(y)} dy \ dz \notag\\
					&=\frac{q}{a^q}\int_{0}^{\infty} y^{q-1} \frac{\gamma(y,bt)}{\Gamma(y)}  dy \notag \\
					\label{newmomLt}
					&=\frac{q b^q}{a^q} \int_0^\infty y^{q-1} \frac{\gamma(by,bt)}{\Gamma(by)}  dy  
					\end{align}
			Since
			\begin{align*}
\frac{\gamma(by,bt)}{\Gamma(by)} \to 1, \quad t \to \infty \ ,
			\end{align*}
			and
			\begin{align*}
\int_0^\infty y^{q-1} dy=\lim_{t\to \infty}\int_0^t y^{q-1} dy \sim \frac{t^q}{q} ,
			\end{align*}
			then from \eqref{newmomLt} we obtain
			\begin{align*}
\mathbf{E}_0[L_t^q] \sim \left(\frac{bt}{a}\right)^q.
			\end{align*}
	\end{itemize}
\end{proof}
\begin{remark}
(Convergence)	In the last theorem we have proved that
	\begin{align*}
\int_0^t   e^{-bs} \mu(bs,q-1,-1) \ ds < \infty 
	\end{align*}
	 for any finite $t>0$. Since \eqref{momenti2} and \eqref{momenti_int} are equivalent, then we gain the absolute convergence on $(0,\infty)$ of the series \eqref{momenti2} .
\end{remark}

	\begin{figure}[h]
	\centering
	\includegraphics[width=5.5cm]{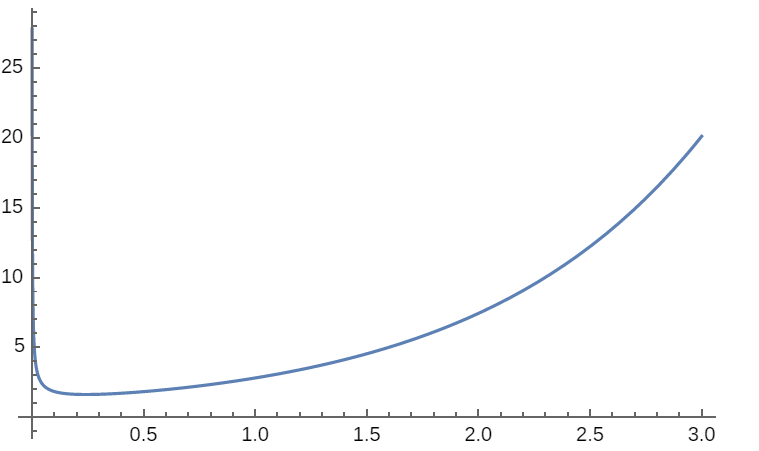}
	\caption{$\nu(x,-1)$ near $x=0$. The picture shows that $\nu(x,-1)$ goes to infinity as x goes to zero. The asymptotic behavior of $\nu(x,\alpha)$ near the origin depends on the sign $\alpha$. }
	\label{nu-1}
\end{figure}
\begin{remark}
	By exploiting the fact that $\mu(x,0,\alpha)=\nu(x,\alpha)$ we provide here a simple manipulation concerned with the first moment of $L_t$.
	From \eqref{ripLt} we know that
	\begin{align*}
	\mathbf{P}_0(L_t > x)=\frac{\gamma(ax,bt)}{\Gamma(ax)} .
	\end{align*}
	Then we can calculate the mean as
	\begin{align*}
	\mathbf{E}_0[L_t]&=\int_{0}^{\infty} \mathbf{P}_0(L_t>x) dx\\
	&=\int_{0}^{\infty} \frac{\gamma(ax,bt)}{\Gamma(ax)} dx \\
	&=\frac{1}{a} \int_{0}^{\infty} \frac{\gamma(y,bt)}{\Gamma(y)} dy \\
	&=\frac{1}{a} \int_{0}^{1} \frac{\gamma(y,bt)}{\Gamma(y)} dy + \frac{1}{a} \int_{1}^{2} \frac{\gamma(y,bt)}{\Gamma(y)} dy + \cdots .
\end{align*}	
By using \eqref{nugammainc} and \eqref{defnu} we obtain		
\begin{align*}
	\mathbf{E}_0[L_t] &=\frac{e^{-bt}}{a} \sum_{n=0}^\infty \nu(bt,n),
	\end{align*}
	which coincides with the result given in \eqref{momenti2} for $q=1$.
\end{remark}
\medskip
Let we focus on the convolution-type operator 
\begin{align*}
\mathfrak{D}_t^\Phi \ u(t)=\int_0^t u'(s) \overline{\Pi}(t-s) ds, \quad u \in AC(0,\infty)
\end{align*}
which has been introduced in \eqref{CapDef}. We notice that
\begin{align*}
\int_0^\infty e^{-\lambda t} \mathfrak{D}_t^\Phi u(t) \ dt = (\lambda \tilde{u} (\lambda) - u(0))\frac{\Phi(\lambda)}{\lambda} = \Phi(\lambda)\tilde{u}(\lambda) - \frac{\Phi(\lambda)}{\lambda} u(0)
\end{align*}
and here we assume that $u(0)\neq 0$. The last steps are justified by the formula \eqref{lapTailLm} and the Laplace transform of a convolution.

Focus now on the Brownian motion $B_t$ time-changed respectively with a Gamma subordinator $H_t$ and its inverse $L_t$. We respectively write $B_{H_t}$ and $B_{L_t}$. The governing equations are well-known. The first case leads to a Markov process with generator $(A,D(A))$ where $A=-\Phi(-\Delta)$ in the sense of Phillips. The second time change can be considered in order to solve the non-local equation
\begin{align*}
\mathfrak{D}^\Phi_t \varphi = \Delta \varphi. 
\end{align*} 
 Since $H_t$ and $L_t$ are independent from $B$, we can write
\begin{align}
\label{momentsVG}
\mathbf{E}_x[(B_{H_t})^q] = \mathbf{E}_0[ \mathbf{E}_x[(B_{H_t})^q | H_t]] = \mathbf{E}_x[(B_1)^q]\, \mathbf{E}_0[(H_t)^{q/2}]
\end{align}
and 
\begin{align}
\label{momentsInvVG}
\mathbf{E}_x[(B_{L_t})^q] = \mathbf{E}_0[ \mathbf{E}_x[(B_{L_t})^q | L_t]] = \mathbf{E}_x[(B_1)^q]\, \mathbf{E}_0[(L_t)^{q/2}].
\end{align}
Formula \eqref{momentsVG} can be obviously written by considering \eqref{momentsH} above. From Theorem \ref{teomomenti} we are now able to write \eqref{momentsInvVG} by replacing $\mathbf{E}_0[(L_t)^\frac{q}{2}]$ with \eqref{momenti2} or \eqref{momenti_int}.

\subsection{Potentials}

To the best of our knowledge, for the Laplace transforms 
\begin{align}
\label{potRep}
\int_0^\infty e^{-c x} l(t,x) dx \quad \text{and} \quad \int_0^\infty e^{-\lambda t} h(t,x) dt
\end{align}
with $c,\lambda >0$, there are no explicit representations. On the other hand, the relation between potentials (\cite[formula 5.5]{CapDov19})
\begin{align*}
\mathbf{E}_x \left[ \int_0^\infty e^{-\lambda t} f(L_t)dt \right] = \frac{\Phi(\lambda)}{\lambda} \mathbf{E}_x \left[ \int_0^\infty e^{-\lambda H_t} f(t) dt \right]
\end{align*}
can be explicitly verified for any symbol $\Phi$. For a general subordinator and the associated inverse process, the problem to obtain the representations \eqref{potRep} is still open. The only case in which we have a closed form for such objects is the case of $\alpha$-stable subordinator and its inverse process. 
That is, we respectively have
\begin{align*}
\mathcal{E}_\alpha (-c t^\alpha) \quad \text{and} \quad x^{\alpha -1} \mathcal{E}_{\alpha, \alpha}(-\lambda x^\alpha) \ ,
\end{align*}
where 
\begin{align*}
\mathcal{E}_{\alpha,\beta}(-z)= \sum_{k=0}^{\infty} \frac{(-z)^k}{\Gamma(\alpha k + \beta)}
\end{align*}
is the Mittag-Leffler function for which $\mathcal{E}_{\alpha}(z)=\mathcal{E}_{\alpha,1}(z)$.

In this section we direct our efforts in order to obtain explicit representations in \eqref{potRep} and  we begin our discussion by introducing the following result. 

\begin{theorem}
	\label{teolapLt}
	Let $u$ be the unique continuous solution on $[0,+\infty)$ to 
	\begin{align}
	\label{eqlapLt}
	\mathfrak{D}_t^\Phi u(t)&=-c\ u(t), \quad u(0)=1, \quad c > 0 \ .
	\end{align}
	Let $L_t$ be the inverse to a gamma subordinator. Then, we have that:
	\begin{itemize}
		\item[(i)] $u(t)=\mathbf{E}_0[e^{-cL_t}]$;
		\item[(ii)] $u(t)$ has the following representation
	\begin{align}
	\label{lapLt2}
	u(t)&=1-\frac{c}{a} e^{-bt} \sum_{n=0}^\infty {e^{\frac{c}{a} n}} \nu(bt e^{-\frac{c}{a}}, n) \\
	\label{LapLt_int}
	&=1- b\frac{c}{a} e^{-\frac{c}{a}} \int_0^t e^{-bs} \nu(bs e^{-\frac{c}{a}},-1) ds
	\end{align}
\end{itemize}
\end{theorem}
\begin{proof}
	\begin{itemize}
		\item[(i)] By applying the $\lambda-$Laplace transform w.r.t. $t$ in \eqref{eqlapLt} and by using 
		$$u(0)=1$$
		we obtain
		\begin{align*}
		\tilde{u}(\lambda) \Phi(\lambda )  + c\ \tilde{u}(\lambda )= \frac{\Phi(\lambda)}{\lambda}.
		\end{align*}
		Then,
		\begin{align*}
		\tilde{u}(\lambda)&=\frac{\Phi(\lambda)}{\lambda} \frac{1}{\Phi(\lambda)+c} = \frac{\Phi(\lambda)}{\lambda} \int_0^\infty e^{-cw} e^{-\Phi(\lambda) w} dw\\
		&=\int_0^\infty e^{-cw}\left(\int_0^\infty e^{-\lambda t} l(t,w) dt\right) dw,
		\end{align*}
		from which
		\begin{align*}
		u(t)=\int_0^\infty e^{-cw} l(t,w) \ dw=\mathbf{E}_0[e^{-cL_t}] \ .
		\end{align*} 	
		From Laplace machinery we have uniqueness. Indeed $u \in \mathcal{M}_0\cap C[0,\infty)$.
\item[(ii)] 
By applying the Laplace transform in the left-hand-side of \eqref{lapLt2}, we write
\begin{align*}
\int_0^\infty e^{-\lambda t} \mathbf{E}_0[e^{-c L_t}] dt &= \int_0^\infty e^{-cx} \frac{\Phi(\lambda)}{\lambda} e^{-x \Phi(\lambda)}=\frac{\Phi(\lambda)}{\lambda} \frac{1}{c+\Phi(\lambda)}, \quad \lambda>0.
\end{align*}
On the other hand
\begin{align*}
&\int_0^\infty e^{-\lambda t} \left[1-\frac{c}{a} e^{-bt} \sum_{n=0}^\infty {e^{\frac{c}{a} n}} \nu(bt e^{-\frac{c}{a}}, n)\right] dt=\\
&=\frac{1}{\lambda} -\frac{c}{a} \sum_{n=0}^\infty {e^{\frac{c}{a} n}} \int_{0}^{\infty} e^{-(\lambda +b) t} \nu(bt e^{-\frac{c}{a}}, n) dt\\
&=\frac{1}{\lambda} -\frac{c}{a} \sum_{n=0}^\infty {e^{\frac{c}{a} n}} \frac{e^\frac{c}{a}}{b} \int_{0}^{\infty} e^{-\frac{\lambda+b}{b}e^\frac{c}{a} t} \nu(t, n) dt,\\
\noalign{by taking into account the formula \eqref{lapnu}, we get that}
&=\frac{1}{\lambda} -\frac{c}{a} \sum_{n=0}^\infty \frac{1}{b} {e^{\frac{c}{a} (n+1)}} \frac{e^{-\frac{c}{a}(n+1) }}{\left(1+\frac{\lambda}{b}\right)^{n+1}} \frac{1}{\ln\left(1+\frac{\lambda}{b}\right) + \frac{c}{a}}\\
&=\frac{1}{\lambda}-\frac{c}{a} \frac{a}{a\ln\left(1+\frac{\lambda}{b}\right) + c} \frac{1}{\lambda}\\
&=\frac{\Phi(\lambda)}{\lambda} \frac{1}{\Phi(\lambda) +c} \ ,
\end{align*}
which is the claim for \eqref{lapLt2}.\\

Now we prove \eqref{LapLt_int}. Since $L_t \geq 0$, for $c >0$ the mean $\mathbf{E}_0[e^{-c L_t} ]$ is well-defined. From the Taylor expansion of the exponential we can write
\begin{align*}
\mathbf{E}_0[e^{-c L_t} ]&=1+ \sum_{k=1}^{\infty}\frac{(-c)^k}{k!}\mathbf{E}_0[L_t^k].
\end{align*}
From \eqref{momenti2} we get the new representation
\begin{align*}
\mathbf{E}_0[e^{-c L_t} ]
&= 1+ \sum_{k=1}^{\infty} \frac{(-c)^k}{a^k} \sum_{n=0}^\infty e^{-bt} \mu(bt,k-1,n) \\
&=1-\frac{c}{a} e^{-bt} \sum_{n=0}^\infty \sum_{k=0}^\infty \left(-\frac{c}{a}\right)^k \mu(bt,k,n)\\
&=[{\text{by using \eqref{serienumu}}}]\\
&=1-\frac{c}{a} e^{-bt} \sum_{n=0}^\infty {e^{\frac{c}{a} n}} \nu(bt e^{-\frac{c}{a}}, n) 
\end{align*}
which has been obtained in \eqref{lapLt2}.\\
By considering \eqref{momenti_int} instead of \eqref{momenti2}, we have
\begin{align*}
\mathbf{E}_0[e^{-c L_t} ]&=1+ \sum_{k=1}^{\infty}\frac{(-c)^k}{k!}\mathbf{E}_0[L_t^k]\\
&=1+ \sum_{k=1}^{\infty}\frac{b}{a^k} {(-c)^k} \int_0^t e^{-bs} \mu(bs,k-1,-1) ds,\\
&=[{\text{by using \eqref{serienumu}}}]\\
&=1-b \frac{c}{a} \int_0^t e^{-bs} e^{-\frac{c}{a}} \nu(bs e^{-\frac{c}{a}},-1) ds\\
&=1-b \frac{c}{a}  e^{-\frac{c}{a}} \int_0^t e^{-bs} \nu(bs e^{-\frac{c}{a}},-1) ds.
\end{align*}
Using the same argument for the convergence of \eqref{momenti_int}, we have that for any finite $t>0$ also the integral in \eqref{LapLt_int} converges.
\end{itemize}
\end{proof}
\begin{remark}
(Convergence)	From the convergence of \eqref{LapLt_int}, and from the equivalence between \eqref{lapLt2} and \eqref{LapLt_int} we have also that the series in \eqref{lapLt2} converges absolutely on $(0,\infty)$.
\end{remark}
Now we move to the next result.

\begin{theorem}
	\label{lemmahtilde}
	For $x \in (0, \infty)$, we have that
\begin{align}
\label{htilde}
\int_0^\infty e^{- \lambda t} h(t,x) dt=e^{-bx} \ \frac{b}{a} \ e^{-\frac{\lambda}{a}} \ \nu(b \ x \ e^{-\frac{\lambda}{a}}, -1), \quad \lambda > 0.
\end{align}
\begin{proof}
	Let us write $\tilde{h}(\lambda,x):=\int_0^\infty e^{- \lambda t} h(t,x) dt$, for positive $x$. Then,
\begin{align*}
\tilde{h}(\lambda,x)&= \int_0^\infty e^{- \lambda t} \frac{b^{at}}{\Gamma(at)} x^{at-1} e^{-bx} dt =\frac{e^{-bx}}{x} \int_0^\infty \frac{e^{-t(\lambda - a \ln b - a \ln x)}}{\Gamma(at)} dt
\\
& = \frac{e^{-bx}}{a x} \int_0^\infty \frac{e^{-\frac{z}{a}(\lambda - a \ln b - a \ln x)}}{\Gamma(z)} dz.
\end{align*}
By using (\cite[formula 6.423]{table})
\begin{align*}
\int_0^\infty e^{-\alpha x} \frac{dx}{\Gamma(x+\beta+1)}=e^{\beta \alpha} \nu(e^{-\alpha},\beta),
\end{align*}
we have
\begin{align}
\label{hnu}
\frac{e^{-bx}}{a x} \int_0^\infty \frac{e^{-\frac{z}{a}(\lambda - a \ln b - a \ln x)}}{\Gamma(z)} dz&= \frac{e^{-bx}}{a x} e^{-\frac{1}{a} (\lambda - a \ln b - a \ln x)} \ \nu(e^{-\frac{1}{a} (\lambda - a \ln b - a \ln x)},-1) \notag \\
&=e^{-bx} \ \frac{b}{a} \ e^{-\frac{\lambda}{a}} \ \nu(b \ x \ e^{-\frac{\lambda}{a}}, -1) .
\end{align}
This concludes the proof.
\end{proof}
\end{theorem}

\begin{proof}[Alternative proof concerning Theorem \ref{thm:Sonine}]
From Theorem \ref{lemmahtilde} we are able to write
		\begin{align*}
		\kappa(x)&=\int_0^\infty h(t,x) dt= \lim_{\lambda \to 0} \int_0^\infty e^{-\lambda t} h(t,x) dt=\\
		&=\lim_{\lambda \to 0} e^{-bx} \ \frac{b}{a} \ e^{-\frac{\lambda}{a}} \ \nu(b \ x \ e^{-\frac{\lambda}{a}}, -1)=e^{-bx} \ \frac{b}{a}   \nu(bx,-1)
		\end{align*}
which gives an alternative proof of the result.
\end{proof}

\section*{Acknowledgment}
The authors thank Sapienza for the support under the Grant Ateneo 2020.

\medskip
\small{
\bibliographystyle{plain} 
\bibliography{invgammaref.bib} 
}
\Addresses
\end{document}